\documentclass{amsart}

\newtheorem{theorem}{Theorem}[section]
\newtheorem{lemma}[theorem]{Lemma}

\theoremstyle{definition}

\theoremstyle{remark}
\newtheorem{remark}[theorem]{Remark}

\numberwithin{equation}{section}

\usepackage{amssymb}
\usepackage{color}
\usepackage{tikz}
\usepackage{mathrsfs}

\usepackage{comment}

\usepackage[figuresright]{rotating}
\usepackage{amsmath}
\usepackage{hyperref}
\usepackage{amsthm}
\usepackage{bm}

\usepackage{fancyhdr}

\usepackage[margin=0.8in]{geometry}

\def\d{\partial}
\def\ddj{\dot \Delta_j}
\def\ddjj{\dot \Delta_{j'}}

\def\dB{{\dot {B}}}

\renewcommand{\div}{\mbox{\rm div}\;\!}
\def\cX{{\mathcal X}}

\def\cC{{\mathcal C}}
\def\tL{\widetilde{L}}
\def\cL{{\mathcal L}}

\def\bu{{\bm{u} }}

\def\bz{{\bm{z} }}
\def\bx{\bm{x}}
\def\bv{{\bm{v} }}

\newcommand\R{\mathbb{R}}

\newcommand\Z{\mathbb{Z}}

\begin{document}

\pagestyle{fancy}
\fancypagestyle{plain}{}
\fancyhf{}
\renewcommand{\footrulewidth}{0mm}
\renewcommand{\headrulewidth}{0mm}
\fancyhead[LE]{\thepage}
\fancyhead[RO]{\thepage}

\title{the global well-posedness of the multi-dimensional compressible Euler system with damping in the $L^p$ critical Besov spaces for $p<2$}

\author{Jianzhong Zhang}

\address{School of Mathematics and Information Science, Shandong Technology and Business University, Yantai, Shandong, 264005,P. R. China}

\email{zhangjz\_91@sdtbu.edu.cn}

\author{Ying Sui}

\address{School of Mathematics and Information Science, Shandong Technology and Business University, Yantai, Shandong, 264005,P. R. China}
\email{suiying@sdtbu.edu.cn}

\author{Xiliang Li*}
\address{School of Mathematics and Information Science, Shandong Technology and Business University, Yantai, Shandong, 264005,P. R. China}

\email{lixiliang@amss.ac.cn}




\begin{abstract}
In this paper, we study the Cauchy's problem of the compressible Euler system with damping and establish the global-in-time well-posedness in $L^p$-type critical Besov spaces for $1\leq p<2$. To achieve it, a new product estimate is established in $L^2$-$L^p$ hybrid Besov spaces.
\end{abstract}

\maketitle

\footnotetext{$^*$ Corresponding author.}

\section{Introduction}
This work studies the compressible Euler equations with linear damping, a model describing gas flow through a porous medium where the solid matrix exerts a frictional force proportional to and opposite to the fluid momentum. The governing system is given by
\begin{eqnarray}\label{1.1}
\left\{\aligned
&\d_t \rho+\div (\rho \bu)=0, \\
&\d_t (\rho \bu)+\nabla \cdot (\rho \bu\otimes \bu)+\nabla P(\rho)+\alpha\rho \bu=0
\endaligned\right.
\end{eqnarray}
for $t\geq0,\ \bx\in \R^d(d\geq 1)$ and a damping coefficient $\alpha>0$
. Here $\rho\in\mathbb{R}^{+}$, $\bu=(u^1, u^2,\cdots, u^d)^\top$ ($\top$ represents transpose) denote the density and velocity of fluid flow, respectively. The pressure function $P=P(\rho)$ is assumed to be smooth around the constant reference density $\bar{\rho}>0$.

There are lots of mathematical results about the existence and asymptotic behavior of system \eqref{1.1} in sobolev space (see \cite{HL,HuangFM2,Wangdehua,Sui,Tan3,XY} ). In critical Besov space, Fang and Xu \cite{FXu} (with improvements in \cite{Jiu}) studied the existence and asymptotic behavior of classical solutions. Later, Xu and Wang \cite{XuWang} justified the relaxation convergence from \eqref{1.1} to the porous medium equation. And also, there are some results on hyperbolic system for balance laws including \eqref{1.1} (see \cite{XuKawaARM14,XuKawaARM15}). Recently, Crin-Barat and Danchin \cite{BaratDanchin,BaratDanchin2,BaratDanchin3} improved those results \cite{XuKawaARM14,XuKawaARM15,XuWang} in the $L^2$-$L^p$ hybrid Besov spaces with $2\leq p\leq\max\{4,\frac{2d}{d+2}\}$, where the low frequencies are bounded in $L^p$-type spaces and the high frequencies in $L^2$-type spaces with a specific linearity assumption. And then, the first author and Xu \cite{XuZhang} remove the assumption and to show that the results of \cite{BaratDanchin,BaratDanchin2} hold true for the general pressure function. Very recently, Crin-Barat and Song \cite{BSong} extend the work of \cite{BaratDanchin,BaratDanchin2,XuZhang} to the case $p\in[2,\infty)$. And a natural question is how about $p<2$?

Introducing a new unknown called ``enthalpy"
$n(\rho)\triangleq \int_1^{\rho} \frac{P'(s)}{s} ds.$ System \eqref{1.1} can be rewritten as
\begin{eqnarray}\label{1.2}
\left\{\aligned
&\d_t n + \bv \cdot  \nabla n + P'(\bar{\rho})\div \bv + G(n)\div \bv = 0 \\
&\d_t \bv+ \bv\cdot \nabla \bv+\nabla n= -\frac{1}{\varepsilon} \bv
\endaligned\right.
\end{eqnarray}
with the initial data
\begin{eqnarray}\label{1.2initial}
 (n,\bv)|_{t=0}=(n_0, \bv_0).
\end{eqnarray}
Here $P'(\bar{\rho})>0$ for $\bar{\rho}>0$, $\bar{\rho}=P'(\bar{\rho})=1$ and the composite function $G(n)$ is smooth.

Define $J_\varepsilon=-[\log_2\varepsilon]+k$ being the threshold between high and low frequencies with a suitable integer $k$ is to be determined later, and the space $E_T^{J_\varepsilon}$ as follows
\begin{equation*}
\begin{split}
&E_T^{J_\varepsilon}=\bigg\{(n,\bv)|(n^{\ell,J_\varepsilon},\bv^{\ell,J_\varepsilon})\in \cC_b([0,T];\dB_{p,1}^{\frac{d}{p}}),\
\varepsilon n^{\ell,J_\varepsilon}\in \tL_T^1(\dB_{p,1}^{\frac{d}{p}+2}),\
\bv^{\ell,J_\varepsilon}\in \tL_T^1(\dB_{p,1}^{\frac{d}{p}+1}),
\varepsilon^{-\frac{1}{2}} \bv\in \tL_T^2(\dB_{p,1}^{\frac{d}{p}}),\\
&\qquad\qquad
(\varepsilon n^{h,J_\varepsilon}, \varepsilon\bv^{h,J_\varepsilon})\in \cC_b([0,T];\dB_{2,1}^{\frac{d}{2}+1}),\
(n^{h,J_\varepsilon},\bv^{h,J_\varepsilon}) \in \tL_T^1(\dB_{2,1}^{\frac{d}{2}+1}),
\varepsilon^{-1}\bv+\nabla n\in \tL_T^1(\dB_{p,1}^{\frac{d}{p}})
\bigg\}.
\end{split}
\end{equation*}
When $T=\infty$ we also use $E_\infty^{J_\varepsilon}$ for convenience.
For the reader's convenience, We introduce the notations $\|\cdot\|_{\dB_{q_1,1}^s}^{h,J_{\varepsilon}}$ and
$\|\cdot\|_{\dB_{q_2,1}^s}^{\ell,J_{\varepsilon}}$ to denote Besov semi-norms with respect to the threshold $J_{\varepsilon}$, that is,
\begin{equation}\label{DefHLB}
\|f\|_{\dB_{q_1,1}^{s_1}}^{h,J_{\varepsilon}}\triangleq \sum_{j\geq J_{\varepsilon}}
2^{s_1j}\|\ddj f\|_{q_1}\quad \mbox{and}\quad
\|f\|_{\dB_{q_2,1}^{s_2}}^{\ell,J_{\varepsilon}}\triangleq \sum_{j\leq J_{\varepsilon}-1}
2^{s_2j}\|\ddj f\|_{q_2}.
\end{equation}
It is not difficult to deduce that for all $\sigma_0>0$,
\begin{equation}\label{HLEst}
\|f\|_{\dB_{q_1,1}^{s_1}}^{h,J_{\varepsilon}}
\leq
2^{-\sigma_0J_{\varepsilon}}\|f\|_{\dB_{q_1,1}^{s_1+\sigma_0}}^{h,J_{\varepsilon}}\quad \mbox{and}\quad
\|f\|_{\dB_{q_1,1}^{s_1}}^{\ell,J_{\varepsilon}}
\leq
2^{\sigma_0J_{\varepsilon}}\|f\|_{\dB_{q_1,1}^{s_1-\sigma_0}}^{\ell,J_{\varepsilon}}.
\end{equation}

The main goal of this paper is to broaden the assumption on $p$ and to show that the results of \cite{BaratDanchin,BaratDanchin2,XuZhang} hold true for the case $p<2$. Our main result is stated as follows.
\begin{theorem}\label{thm}
Assume $1\leq p <2$, $d\geq 1$ and the pressure $P$ satisfies $P'(\bar{\rho})>0$ for $\bar{\rho}>0$. There is a small constant $\delta_1>0$ such that if
$
\|(n_0,\bv_0)\|_{\dB_{p,1}^{\frac{d}{p}}}^{\ell,J_\varepsilon}
+\varepsilon\|(n_0,\bv_0)\|_{\dB_{2,1}^{\frac{d}{2}+1}}^{h,J_\varepsilon}
\leq \delta_1,
$
then the Cauchy problem \eqref{1.2}-\eqref{1.2initial} admits a uniform global unique solution in $E_\infty^{J_\varepsilon}$ for all $\varepsilon>0$.
\end{theorem}
\begin{remark}
Note that $\dot{B}_{p,1}^{\frac{d}{p}}\hookrightarrow \dot{B}_{2,1}^{\frac{d}{2}}$ when $p<2$, that means the existence and uniqueness of solutions hold in a smaller space rather than escaping to the complement $\dot{B}_{2,1}^{\frac{d}{2}}\setminus \dot{B}_{p,1}^{\frac{d}{p}}$. On the other hand, \eqref{DefHLB} implies that a smaller space for $p$ corresponds to more singular and concentrated data, which naturally arises in many physical problems (such as point sources and vortex filaments).
\end{remark}

\begin{remark}
Based on Theorem \ref{thm}, the relaxation limit for the case $p<2$ can be performed by the same procedure as \cite{BaratDanchin2}.
\end{remark}

The major difficulty of proof of Theorem \ref{thm} lies in dealing with the nonlinear terms. To achieve it, a new product estimation will be developed, see Lemma \ref{NonClassicalProLaw1} below. As a matter of fact, the new tool could be applied to investigate other systems, such as the hyperbolic-parabolic chemotaxis system (see \cite{BaratShou2}) in $L^2-L^p$ framework and Navier-Stokes equations (see \cite{Danchin} ) in $L^p-L^2$ framework for $p<2$.

\section{Global a priori estimates}
In this section, we only give the key a priori estimate for the case $p<2$, which lead to the global existence and uniqueness of solutions in $E_\infty^{J_\varepsilon}$. See \cite{BShouZhang} for more details. For convenience, we use $ f\lesssim g $ to denote that there exists a generic constant $C>0$ independent on $\varepsilon$ such that $f\leq C g$ in this section.

First of all, we establish the following product estimation in $L^p-L^2$ hybrid Besov spaces.
\begin{lemma}\label{NonClassicalProLaw1}
Let $0<s_1<\frac{d}{2}$, $1\leq p<2$. Then, we have the following inequality
\begin{align}\label{newproduct}
\|ab\|_{\dB_{p,1}^{s_1}}^{\ell,J_\varepsilon}
\leq C
\left(\|b\|_{\dB_{p,1}^{s_1}}^{\ell,J_\varepsilon}+ 2^{(s_1-\frac{d}{p})J_\varepsilon}\|b\|_{\dB_{2,1}^{\frac{d}{2}}}^{h,J_\varepsilon} \right)\|a\|_{\dB_{2,1}^{\frac{d}{2}}}
\end{align}
with $C$ independent on $\varepsilon$.
\end{lemma}

\begin{proof}
Using Bony's paraproduct decomposition, we have
\begin{equation}\nonumber
\begin{aligned}
ab=T_{a}b+R[a,b]+T_{b}a\quad \mbox{with}\quad
T_{a}b\triangleq\sum_{j'\in \Z} \dot S_{j'-1}a \ddjj b\quad\mbox{and}\quad
R[a,b]\triangleq\sum_{|j'-j''|\leq 1} \dot\Delta_{j''} a \ddjj b.
\end{aligned}
\end{equation}

First, we bound $T_{a}b$. It is clear that
\begin{equation}\nonumber
\begin{aligned}
\|T_{a} b\|_{\dB_{p,1}^{s_1}}^{\ell,J_\varepsilon}
\leq
\sum_{j\leq J-2\atop |j-j'|\leq 1}2^{s_1 j} \|\dot S_{j'-1} a \ddj \ddjj b\|_{L^p}
+\sum_{j\leq J-2\atop |j-j'|\leq 4} 2^{s_1 j}\|[\ddj, \dot S_{j'-1} a] \ddjj b\|_{L^p}.
\end{aligned}
\end{equation}
The embedding $\dot B_{p,1}^{\frac{d}{p}}\hookrightarrow L^{\infty}$ leads to
\begin{equation}\nonumber
\begin{aligned}
\sum_{j\leq J-2\atop |j-j'|\leq 1} 2^{s_1j}\|\dot S_{j'-1} a \ddj \ddjj b\|_{L^p}
\lesssim
\|\dot S_{j'-1} a\|_{L^\infty}\sum_{j\leq J-2} 2^{s_1j} \|\ddj b\|_{L^p}
\lesssim
\|a\|_{\dB_{2,1}^{\frac{d}{2}}}\|b\|_{\dB_{p,1}^{s_1}}^{\ell,J_\varepsilon}.
\end{aligned}
\end{equation}
Note that
\begin{equation}\nonumber
\begin{aligned}
\sum_{j\leq J-2\atop |j-j'|\leq 4} 2^{s_1 j}\|[\ddj, \dot S_{j'-1} a] \ddjj b\|_{L^p}
\lesssim\Big(\sum\limits_{j'\leq J-2}2^{s_1j'}
+ \sum\limits_{J -2 \leq j'\leq J+2} \Big)2^{s_1j'}
\|[\ddj, S_{j'-1} a] \Delta_{j'} b\|_{L^p}.
\end{aligned}
\end{equation}
The Young-like inequality ensures that the commutator estimate in \cite[Lemma 2.97]{chemin} also holds for $[\ddj, S_{j'-1} a] \Delta_{j'} b$
and therefore we have
\begin{equation}\nonumber
\begin{aligned}
\sum\limits_{j'\leq J-2}2^{s_1j'}
\|[\ddj, S_{j'-1} a] \Delta_{j'} b\|_{L^p}
\lesssim
\|\nabla a\|_{\dB_{\infty,1}^{-1}}
\|b\|_{\dB_{p,1}^{s_1}}^{\ell,J_\varepsilon}
\lesssim
\|a\|_{\dB_{2,1}^{\frac{d}{2}}}
\|b\|_{\dB_{p,1}^{s_1}}^{\ell,J_\varepsilon}.
\end{aligned}
\end{equation}
Similarly, as $\frac{2p}{2-p}\leq 2$ when $1\leq p\leq 2$, one has
\begin{equation}\nonumber
\begin{aligned}
&\sum\limits_{J -2 \leq j'\leq J+2}2^{s_1j'}
\|[\ddj, S_{j'-1} a] \Delta_{j'} b\|_{L^p}
\lesssim
2^{(s_1-\frac{d}{p})J_\varepsilon}
\|b\|_{\dB_{2,1}^{\frac{d}{2}}}^{h,J_\varepsilon} \|a\|_{\dB_{\frac{2p}{2-p},1}^{\frac{d}{p}-\frac{d}{2}}}^{h,J_\varepsilon}\lesssim 2^{(s_1-\frac{d}{p})J_\varepsilon}
\|b\|_{\dB_{2,1}^{\frac{d}{2}}}^{h,J_\varepsilon} \|a\|_{\dB_{2,1}^{\frac{d}{2}}}^{h,J_\varepsilon}.
\end{aligned}
\end{equation}
Hence, it follows that
\begin{equation}\nonumber
\begin{aligned}
\|T_{a} b\|_{\dB_{p,1}^{s_1}}^{\ell,J_\varepsilon} \lesssim \|a\|_{\dB_{2,1}^{\frac{d}{2}}}\|b\|_{\dB_{p,1}^{s_1}}^{\ell,J_\varepsilon}
+2^{(s_1-\frac{d}{p})J_\varepsilon}\|b\|_{\dB_{2,1}^{\frac{d}{2}}}^{h,J_\varepsilon} \|a\|_{\dB_{2,1}^{\frac{d}{2}}}^{h,J_\varepsilon}
\lesssim
\left(\|b\|_{\dB_{p,1}^{s_1}}^{\ell,J_\varepsilon}+ 2^{(s_1-\frac{d}{p})J_\varepsilon}\|b\|_{\dB_{2,1}^{\frac{d}{2}}}^{h,J_\varepsilon} \right)\|a\|_{\dB_{2,1}^{\frac{d}{2}}}.
\end{aligned}
\end{equation}

For $T_{b} a$, because $s_1-\frac{d}{2}<0$, it follows the classical paraproduct estimate (see \cite{Danchin}) that
\begin{equation}\nonumber
\begin{aligned}
\|T_{b} a\|_{\dB_{p,1}^{s_1}}^{\ell,J_\varepsilon}
\lesssim
\|b\|_{\dB_{\frac{2p}{2-p},1}^{s_1-\frac{d}{2}}}\|a\|_{\dB_{2,1}^{\frac{d}{2}}}
\lesssim
\|b\|_{\dB_{2,1}^{s_1+\frac{d}{2}-\frac{d}{p}}}\|a\|_{\dB_{2,1}^{\frac{d}{2}}}
\lesssim
\left(\|b\|_{\dB_{p,1}^{s_1}}^{\ell,J_\varepsilon}+ 2^{(s_1-\frac{d}{p})J_\varepsilon}\|b\|_{\dB_{2,1}^{\frac{d}{2}}}^{h,J_\varepsilon} \right)\|a\|_{\dB_{2,1}^{\frac{d}{2}}}.
\end{aligned}
\end{equation}

Finally, by classical remainder estimates (see \cite[Theorem 2.85]{chemin}) we can directly obtain
\begin{equation}\nonumber
\begin{aligned}
\|R[a,b]\|_{\dB_{p,1}^{s_1}}^{\ell,J_\varepsilon}
\lesssim
\|b\|_{\dB_{\frac{2p}{2-p},1}^{s_1-\frac{d}{2}}}\|a\|_{\dB_{2,1}^{\frac{d}{2}}}
\lesssim
\left(\|b\|_{\dB_{p,1}^{s_1}}^{\ell,J_\varepsilon}+ 2^{(s_1-\frac{d}{p})J_\varepsilon}\|b\|_{\dB_{2,1}^{\frac{d}{2}}}^{h,J_\varepsilon} \right)\|a\|_{\dB_{2,1}^{\frac{d}{2}}}.
\end{aligned}
\end{equation}
Adding above three inequality together, we can finally obtain \eqref{newproduct}.
\end{proof}
For simplicity, we define $\cX(T):=\|(n,\bv)\|_{E_T^{J\varepsilon}}$. The proof of Theorem \ref{thm} reduces to establishing a global-in-time a priori estimate. Specifically, we claim that if
\begin{equation*}
\begin{split}
\|n\|_{L^\infty}+\|\bv\|_{L^\infty} \ll 1 \quad \text{on} \quad [0,T],
\end{split}
\end{equation*}
then there exists a constant $C$, independent of $T$ and $\varepsilon$, such that
\begin{equation}\label{APrioriEstimate}
\cX(T)\leq C\big( \cX(0)+\cX^2(T)\big) \qquad\text{for all } \varepsilon>0.
\end{equation}

The proof of the inequality \eqref{APrioriEstimate} is divided into two steps.

\vspace{0.5pc}
{\bf Step 1: The low-frequency estimate in the $L^p$ framework}
\vspace{0.5pc}

Defining the effective velocity as $\bz\triangleq \varepsilon^{-1}\bv+\nabla n$ (cf. the analogous case for Navier-Stokes equations in \cite{Haspot2}), and operating on (\ref{1.2}) with $\ddj$, yields
\begin{equation}\label{2.1}
\begin{split}
\left\{\aligned
&\d_t \ddj n - \varepsilon\Delta \ddj n =- \varepsilon\div\ddj\bz- \ddj (\bv\cdot \nabla n)-\ddj(G(n)\div \bv)\\
&\d_t\ddj \bz+\frac{1}{\varepsilon} \ddj\bz = H
\endaligned\right.
\end{split}
\end{equation}
with
$$H=\varepsilon
\left(\nabla\Delta \ddj n-\nabla \div\ddj\bz\right)
- \ddj \nabla(\bv\cdot \nabla n)
-\nabla\ddj(G(n)\div \bv)
-\frac{1}{\varepsilon}\ddj(\bv\cdot\nabla\bv).$$
It is clear that the equation on $n$ is a heat equation and the equation on $\bz$ is a damped equation, hence the standard estimate (see \cite{BShouZhang,XuZhang}) implies that
\begin{equation}\label{LowEst3}
\begin{split}
&\|(n,\varepsilon \bz)\|_{\tL_T^\infty(\dot B_{p,1}^{\frac{d}{p}})}^{\ell,J_\varepsilon}
+\bigg(\varepsilon\|n\|_{\tL_T^1(\dot B_{p,1}^{\frac{d}{p}+2})}^{\ell,J_\varepsilon}
+\|\bz\|_{\tL_T^1(\dot B_{p,1}^{\frac{d}{p}})}^{\ell,J_\varepsilon}\bigg)\\
&\quad\leq
C\bigg(\|(n_0, \varepsilon \bz_0)\|_{\dot B_{p,1}^{\frac{d}{p}}}^{\ell,J_\varepsilon}
+(1+2^{k})\|(\bv\cdot \nabla n,G(n)\div \bv,\bv\cdot \nabla\bv)\|_{\tL_T^1(\dB_{p,1}^{\frac{d}{p}})}^{\ell,J_\varepsilon}\bigg)
\end{split}
\end{equation}
with $C$ depending only on $d,p,k$.

Next, we will deal with the nonlinear part. For the first nonlinear term, it follows the Lemma \ref{NonClassicalProLaw1} that
\begin{equation*}
\begin{split}
\|\bv\cdot \nabla n\|_{\tL_T^1(\dB_{p,1}^{\frac{d}{p}})}^{\ell,J_\varepsilon}
&\lesssim
\left(
\|\nabla n\|_{\tL_T^2(\dB_{p,1}^{\frac{d}{p}})}^{\ell,J_\varepsilon}
+ \|\nabla n\|_{\tL_T^2(\dB_{2,1}^{\frac{d}{2}})}^{h,J_\varepsilon}
\right)
\|\bv\|_{\tL_T^2(\dB_{2,1}^{\frac{d}{2}})}.
\end{split}
\end{equation*}
Then the interpolation inequalities (see Lemma 3.2 in \cite{XuZhang}) implies
\begin{equation}\label{NL1}
\begin{split}
\|\bv\cdot \nabla n\|_{\tL_T^1(\dB_{p,1}^{\frac{d}{p}})}^{\ell,J_\varepsilon}
&\lesssim
\left(\left(\| n\|_{\tL_T^\infty(\dB_{p,1}^{\frac{d}{p}})}^{\ell,J_\varepsilon}
\|n\|_{\tL_T^1(\dB_{p,1}^{\frac{d}{p}+2})}^{\ell,J_\varepsilon}\right)^{\frac{1}{2}}
+\left(\|n\|_{\tL_T^\infty(\dB_{2,1}^{\frac{d}{2}+1})}^{h,J_\varepsilon}
\|n\|_{\tL_T^1(\dB_{2,1}^{\frac{d}{2}+1})}^{h,J_\varepsilon}\right)^{\frac{1}{2}}
\right)\\
&\quad
\times\left(\|\bv\|_{\tL_T^2(\dB_{p,1}^{\frac{d}{p}})}^{\ell,J_\varepsilon}
+\varepsilon\left(\|\bv\|_{\tL_T^\infty(\dB_{2,1}^{\frac{d}{2}+1})}^{h,J_\varepsilon}
\|\bv\|_{\tL_T^1(\dB_{2,1}^{\frac{d}{2}+1})}^{h,J_\varepsilon}\right)^{\frac{1}{2}}\right)\\
&\lesssim
\cX^2(T)+\cX(T)\varepsilon^{-\frac{1}{2}}\|\bv\|_{\tL_T^2(\dB_{p,1}^{\frac{d}{p}})}^{\ell,J_\varepsilon}.
\end{split}
\end{equation}
And similarly, for the convection term we can deduce
\begin{equation}\label{NL2}
\begin{split}
\|\bv\cdot \nabla \bv\|_{\tL_T^1(\dB_{p,1}^{\frac{d}{p}})}^{\ell,J_\varepsilon}
\lesssim
\cX^2(T).
\end{split}
\end{equation}
For the last nonlinear term, the inequality \eqref{newproduct} implies
\begin{equation*}
\begin{split}
\|G(n)\div \bv\|_{\tL_T^1(\dB_{p,1}^{\frac{d}{p}})}^{\ell,J_\varepsilon}
&\lesssim
\|G(n)\|_{\tL_T^\infty(\dB_{2,1}^{\frac{d}{2}})}
\left(\|\div \bv\|_{\tL_T^1(\dB_{p,1}^{\frac{d}{p}})}^{\ell,J_\varepsilon}
+ \|\div \bv\|_{\tL_T^1(\dB_{2,1}^{\frac{d}{2}})}^{h,J_\varepsilon}\right).
\end{split}
\end{equation*}
By the classical composition estimate (see \cite{chemin}), we can obtain
\begin{equation}\label{NL3}
\begin{split}
\|G(n)\div \bv\|_{\tL_T^1(\dB_{p,1}^{\frac{d}{p}})}^{\ell,J_\varepsilon}
\lesssim
\|n\|_{\tL_T^\infty(\dB_{2,1}^{\frac{d}{2}})}
\left(\|\bv\|_{\tL_T^1(\dB_{p,1}^{\frac{d}{p}+1})}^{\ell,J_\varepsilon} +\|\bv\|_{\tL_T^1(\dB_{2,1}^{\frac{d}{2}+1})}^{h,J_\varepsilon}\right)
\lesssim
\cX^2(T).
\end{split}
\end{equation}

Inserting \eqref{NL1}-\eqref{NL3} into \eqref{LowEst3}, we can obtain
\begin{equation*}
\begin{split}
&\|(n, \varepsilon \bz)\|_{\tL_T^\infty(\dot B_{p,1}^{\frac{d}{p}})}^{\ell,J_\varepsilon}
+\bigg(\varepsilon\|n\|_{\tL_T^1(\dot B_{p,1}^{\frac{d}{p}+2})}^{\ell,J_\varepsilon}
+\|\bz\|_{\tL_T^1(\dot B_{p,1}^{\frac{d}{p}})}^{\ell,J_\varepsilon}\bigg)\\
&\quad\leq C\bigg(
\|(n_0, \varepsilon \bz_0)\|_{\dot B_{p,1}^{\frac{d}{p}}}^{\ell,J_\varepsilon}
+\cX^2(T)+\cX(T)\varepsilon^{-\frac{1}{2}}
\|\bv\|_{\tL_T^2(\dB_{p,1}^{\frac{d}{p}})}^{\ell,J_\varepsilon}\bigg) ,
\end{split}
\end{equation*}
which eventually leads to
\begin{equation}\label{LowEst4}
\begin{split}
&\|n\|_{\tL_T^\infty(\dot B_{p,1}^{\frac{d}{p}})}^{\ell,J_\varepsilon}
+\|\bv\|_{\tL_T^\infty(\dot B_{p,1}^{\frac{d}{p}})}^{\ell,J_\varepsilon}
+\bigg(\varepsilon\|n\|_{\tL_T^1(\dot B_{p,1}^{\frac{d}{p}+2})}^{\ell,J_\varepsilon}+\|\varepsilon^{-1}\bv+\nabla n\|_{\tL_T^1(\dot B_{p,1}^{\frac{d}{p}})}^{\ell,J_\varepsilon}
+\|\bv\|_{\tL_T^1(\dot B_{p,1}^{\frac{d}{p}+1})}^{\ell,J_\varepsilon}
+\varepsilon^{-\frac{1}{2}}
\|\bv\|_{\tL_T^2(\dB_{p,1}^{\frac{d}{p}})}^{\ell,J_\varepsilon}\bigg)\\
&\quad\lesssim
\cX(0)+\cX^2(T).
\end{split}
\end{equation}

It is known that most hyperbolic systems with no $0$-order terms are ill-posed in $L^p$ spaces with $p\neq2$ (see \cite{Brenner}), and the
the $0$-th order damping plays a key role in above Low-frequency analysis.
Next, we turn to the high-frequency analysis. In that case, the formal eigenvalue analysis (see \cite{BaratDanchin}) implies that it is only suitable within the $L^2$ framework. Therefore, we use the energy method which is a classical method in dissipative system (such as \cite{Zuazua,XuKawaARM14,Li,LiZ})

\vspace{0.5pc}
{\bf Step 2: The high-frequency estimates in the $L^2$ framework}
\vspace{0.5pc}

First, we localize system (\ref{1.2}) by applying the operator $\ddj$ and obtaining
\begin{eqnarray*}
\left\{\aligned
&\d_t \ddj n + \div \ddj\bv + G(n)\div \ddj \bv+ \bv \cdot  \nabla \ddj n  = R_j^1+R_j^2 \\
&\d_t\ddj \bv +\nabla \ddj n+ \bv\cdot \nabla \ddj\bv+\frac{1}{\varepsilon} \ddj \bv= R_j^3
\endaligned\right.
\end{eqnarray*}
with the commutators defined as
\begin{eqnarray*}
\begin{split}
R_j^1:=-\ddj(\bv \cdot  \nabla n)+\bv \cdot  \nabla \ddj n,\
R_j^2:=-\ddj(G(n)\div \bv)+  G(n)\div \ddj \bv,\
R_j^3:=-\ddj(\bv\cdot \nabla\bv)+\bv\cdot \nabla \ddj\bv.
\end{split}
\end{eqnarray*}

By using the $L^2$ weighted energy method, as in \cite{BaratDanchin, XuZhang}, one can deduce that
\begin{equation}\label{dissi4}
\begin{split}
&\frac{1}{2}\frac{d}{dt}\varepsilon\cL_j^2+\|(\nabla\ddj n,\nabla\ddj \bv)\|_{L^2}^2\\
&\quad\lesssim
 \varepsilon\|(\d_tG(n),\nabla G(n), \nabla \bv)\|_{L^\infty}\cL_j^2
  +\cL_j \varepsilon \sum_{i=1}^32^j\|R_j^i\|_{L^2}
\end{split}
\end{equation}
with the Lyapunov functional $\cL_j^2$
$$\varepsilon\cL_j^2:= 2^{2j}\varepsilon\|(\ddj n, \ddj \bv)\|_{L^2}^2+2\tilde{c} \int \ddj \bv \cdot\nabla\ddj n dx
\approx 2^{2j}\varepsilon\|(\ddj n, \ddj \bv)\|_{L^2}^2\quad\mbox{for}\quad \tilde{c}2^{-k}<1.$$
And then it follows from the Lemma 5.1. in \cite{BaratDanchin2} that
\begin{equation}\label{HighEst}
\begin{split}
&\varepsilon\|(n, \bv)\|_{\tL_T^\infty(\dB_{2,1}^{\frac{d}{2}+1})}^{h,J_\varepsilon}
+\|(n, \bv)\|_{\tL_T^1(\dB_{2,1}^{\frac{d}{2}+1})}^{h,J_\varepsilon}\\
&\quad\lesssim
\cX(0)
+\varepsilon \|(\d_t G(n), \nabla G(n), \nabla \bv)\|_{L_T^\infty(L^\infty)}\cX(T)
+\varepsilon\int_0^T \sum_{i=1}^3\sum_{j\geq J_\varepsilon}2^{(\frac{d}{2}+1)j}\|R_j^i\|_{L^2}dt.
\end{split}
\end{equation}

Since $2^{J_\varepsilon}\approx 2^{k}\varepsilon^{-1}$, by the spatial embedding
$\dB_{p,1}^{\frac{d}{p}}(\R^d)\hookrightarrow \dB_{2,1}^{\frac{d}{2}}(\R^d)\hookrightarrow L^\infty(\R^d)$, the classical estimation on smooth functions (see Corollary 2.65 in \cite{chemin}) and \eqref{HLEst} we can obtain that
\begin{equation}\label{HNonLinEst1}
\begin{split}
\varepsilon\|(\nabla G(n),\nabla \bv)\|_{L_T^\infty(L^\infty)}
\lesssim
\varepsilon\|(n, \bv)\|_{\tL_T^\infty(\dB_{2,1}^{\frac{d}{2}+1})}
\lesssim\cX(T).
\end{split}
\end{equation}
By the continuity equation, it is not difficult to get
\begin{equation}\label{HNonLinEst2}
\begin{split}
\varepsilon\|\d_t G(n)\|_{L_T^\infty(L^\infty)}
&\lesssim
\varepsilon \|\div \bv+\bv \cdot  \nabla n
+ G(n)\div\bv\|_{\tL_T^\infty(\dB_{2,1}^{\frac{d}{2}})}\\
&\lesssim
\varepsilon\|\bv\|_{\tL_T^\infty(\dB_{2,1}^{\frac{d}{2}+1})}
+\|\bv\|_{\tL_T^\infty(\dB_{2,1}^{\frac{d}{2}})}\varepsilon \|\nabla n\|_{\tL_T^\infty(\dB_{2,1}^{\frac{d}{2}})}\
+\|n\|_{\tL_T^\infty(\dB_{2,1}^{\frac{d}{2}})}\varepsilon \|\div \bv\|_{\tL_T^\infty(\dB_{2,1}^{\frac{d}{2}})}\\
&\lesssim
\cX(T)+\cX^2(T).
\end{split}
\end{equation}

Next, we bound the commutator $R_j^i$ ($i=1,2,3$), by the Lemma 2.100 in \cite{chemin} we can deduce
\begin{equation}\label{HNonLinEst3}
\begin{split}
&\varepsilon\int_0^T \sum_{i=1}^3\sum_{j\geq J_\varepsilon}2^{(\frac{d}{2}+1)j}
\|R_j^i\|_{L^2}dt
\lesssim
\varepsilon\|\bv\|_{\tL_T^1(\dB_{2,1}^{\frac{d}{2}+1})}
\left(\|n\|_{\tL_T^\infty(\dB_{2,1}^{\frac{d}{2}+1})}
+\|\bv\|_{\tL_T^\infty(\dB_{2,1}^{\frac{d}{2}+1})}
\right)
\lesssim
\cX^2(T).
\end{split}
\end{equation}

Substitute \eqref{HNonLinEst1}-\eqref{HNonLinEst3} into \eqref{HighEst} we can finally obtain
\begin{equation}\label{HighEst2}
\begin{split}
&\varepsilon\|(n, \bv)\|_{\tL_T^\infty(\dB_{2,1}^{\frac{d}{2}+1})}^{h,J_\varepsilon}
+\|(n, \bv)\|_{\tL_T^1(\dB_{2,1}^{\frac{d}{2}+1})}^{h,J_\varepsilon}
+\varepsilon^{-\frac{1}{2}}\|\bv\|_{\tL_T^2(\dB_{p,1}^{\frac{d}{p}})}^{h,J_\varepsilon}\\
&\quad+|\varepsilon^{-1}\bv+\nabla n\|_{\tL_T^1(\dot B_{p,1}^{\frac{d}{p}})}^{h,J_\varepsilon}
\lesssim
\cX(0)+  \cX^2(T)+\cX^3(T).
\end{split}
\end{equation}

The estimate \eqref{APrioriEstimate} follows immediately from a combination of \eqref{HighEst2} and \eqref{LowEst4}. Applying a standard bootstrap argument, as detailed in \cite{BShouZhang}, establishes the global well-posedness of the solution. Consequently, Theorem \ref{thm} is proved.

\section*{Acknowledgments}
J. Z. Zhang was supported by the Natural Science Foundation of Shandong Province , China (ZR2024QA003); X. L. Li was supported by the National
Natural Science Foundation of China (12571178).

\bibliographystyle{amsplain}

\end{document}